\documentclass[11pt]{amsart}

\usepackage{xypic}
\usepackage{hyperref}

\usepackage{latexsym}   
\usepackage{amssymb}    
\usepackage{amsmath}    
\usepackage{amsthm}     
\usepackage{hyperref}
\usepackage{mathpple}

\numberwithin{equation}{subsection}
\newtheorem{theorem}{Theorem}[section]
\newtheorem{lemma}[theorem]{Lemma}
\newtheorem{proposition}[theorem]{Proposition}
\newtheorem{corollary}[theorem]{Corollary}
\newtheorem{application}[theorem]{Application}
\newtheorem*{theoremprank}{Theorem \ref{Tprankmono}}
\newtheorem*{corboundary}{Corollary \ref{Cchainelliptic}}

\theoremstyle{definition}
\newtheorem{remark}[theorem]{Remark}
\newtheorem{question}[theorem]{Question}

\renewenvironment{proof}{\par\medbreak
 \textit{Proof.}\hskip.5em\ignorespaces}{$\Box$\medbreak}

\newcommand{\invlim}[1]{\lim_{\stackrel{\leftarrow}{#1}}}
\newcommand{\til}[1]{{\widetilde{#1}}}

\def\calm{{\mathcal M}}
\def\caln{{\mathcal N}}
\newcommand{\oneover}[1]{\frac{1}{#1}}

\def\geom{{\rm geom}}
\def\sm{{\rm sm}}

\def\et{{\text{\'et}}}

\def\del{{\partial}}

\def\mmu{{\pmb \mu}}
\newcommand{\st}[1]{\{#1\}}
\def\ra{\rightarrow}

\global\let\hom\undefined
\DeclareMathOperator{\hom}{Hom}
\DeclareMathOperator{\spec}{Spec}
\DeclareMathOperator{\aut}{Aut}
\DeclareMathOperator{\gl}{GL}
\DeclareMathOperator{\pic}{Pic}
\def\sp{{\mathop{\rm Sp}}}
\DeclareMathOperator{\gsp}{GSp}
\DeclareMathOperator{\Isom}{Isom}
\DeclareMathOperator{\End}{End}
\def\std{{\mathop{\rm std}}}
\DeclareMathOperator{\gal}{Gal}
\def\tensor{\otimes}
\def\integ{\mathbb Z}
\def\gp{\mathbb G}

\def\proj{\mathbb P}
\newcommand{\abs}[1]{{\left|#1\right|}}
\newcommand{\rest}[1]{|_{#1}}
\def\rat{\mathbb Q}
\def\ff{\mathbb F}
\def\inject{\hookrightarrow}
\def\cross{\times}
\def\units{^\cross}
\DeclareMathOperator{\id}{id}
\def\iso{\cong}
\newcommand{\ang}[1]{{{\langle #1 \rangle}}}
\def\mono{{\sf M}}
\def\cala{{\mathcal A}}
\def\calc{{\mathcal C}}
\def\calh{{\mathcal H}}
\def\calo{{\mathcal O}}
\def\cals{{\mathcal S}}
\def\calt{{\mathcal T}}
\def\inv{^{-1}}

\renewcommand{\bar}[1]{{\overline{#1}}}
\newcommand{\floor}[1]{{\lfloor #1 \rfloor}}

\newenvironment{alphabetize}{\begin{enumerate}

}{\end{enumerate}}

\begin{document}


\title{Monodromy of the $p$-rank strata of the moduli space of curves}
\author{Jeffrey D. Achter, Rachel Pries}
\address{Department of Mathematics, Colorado State University, Fort Collins, CO 80523} 

\begin{abstract}
\noindent
We determine the $\integ/\ell$-monodromy and $\integ_\ell$-monodromy of
every irreducible component of the stratum $\calm_g^f$ of curves
of genus $g$ and $p$-rank $f$ in characteristic $p$.  In particular,
we prove that the $\integ/\ell$-monodromy of every component of
$\calm_g^f$ is the symplectic group $\sp_{2g}(\integ/\ell)$ if $g \geq
3$ and if $\ell$ is a prime distinct from $p$.  The method involves
results on the intersection of $\bar\calm_g^f$ with the boundary of
$\bar\calm_g$.  
We give applications to the generic
behavior of automorphism groups, Jacobians, class groups, and zeta
functions of curves of given genus and $p$-rank.  
\end{abstract}


\maketitle

\section{Introduction}

Suppose $C$ is a smooth connected projective curve of genus $g \geq 1$ over
an algebraically closed field $k$ of characteristic $p>0$.  The
Jacobian $\pic^0(C)$ is a principally polarized abelian variety of
dimension $g$.  The number of $p$-torsion points of $\pic^0(C)$ is $p^f$
for some integer $f$, called the $p$-rank of $C$, with $0 \le f \le g$. 

Let $\calm_g$ be the moduli space over $k$  of smooth connected projective
curves of genus $g$; it is a smooth Deligne-Mumford stack over $k$.
The $p$-rank induces a stratification $\calm_g= \cup \calm_g^f$ by
locally closed reduced substacks $\calm_g^f$,
whose geometric points correspond to curves of genus $g$ and $p$-rank $f$.     

Let $\ell$ be a prime number distinct from $p$.  In this paper, we compute
the $\ell$-adic monodromy of every irreducible component of
$\calm_g^f$.  The main result implies that there is no restriction on
the monodromy group other than that it preserve the symplectic pairing
coming from the principal polarization.  Heuristically, this means
that $p$-rank constraints alone do not force the existence of
extra automorphisms (or other algebraic cycles) on a family of curves.

To describe this result more precisely, let $S$ be a connected stack
over $k$, and let $s$ be a geometric point of $S$.  Let $C \ra S$ be a
relative smooth proper curve of genus $g$ over $S$.  Then
$\pic^0(C)[\ell]$ is an \'etale cover of $S$ with geometric fiber
isomorphic to $(\integ/\ell)^{2g}$.  The fundamental group
$\pi_1(S,s)$ acts linearly on the fiber $\pic^0(C)[\ell]_s$, and the
monodromy group $\mono_\ell(C \ra S, s)$ is the image of $\pi_1(S,s)$
in $\aut(\pic^0(C)[\ell]_s)$.  For the main result we determine
$\mono_\ell(S):=\mono_\ell(C \to S,s)$, where $S$ is an irreducible
component of $\calm_g^f$ and $C \ra S$ is the tautological curve.
This also determines the $\ell$-adic monodromy group $\mono_{\integ_\ell}(S)$.

\begin{theoremprank}
Let $\ell$ be a prime distinct from $p$ and suppose $g \ge 1$.
Suppose $0 \leq f \leq g$, and $f \not = 0$ if $g \leq 2$. 
Let $S$ be an irreducible component of $\calm_g^f$, the $p$-rank $f$ 
stratum in $\calm_g$. 
Then $\mono_\ell(S)\iso \sp_{2g}(\integ/\ell)$ and
$\mono_{\integ_\ell}(S) \iso \sp_{2g}(\integ_\ell)$.
\end{theoremprank}

We also prove an analogous result about $p$-adic monodromy
(Proposition \ref{proppadic}).

We give four applications of Theorem \ref{Tprankmono} in Section \ref{secapp}.
The first two do not use the full strength of the theorem, 
in that they can be deduced solely from knowledge of the $\rat_\ell$-monodromy.
Application (i) complements \cite[Thm.\ 1]{poonennoextra1}
(and recovers \cite[Thm.\ 1.1(i)]{achterglasspries}), while 
application (ii) complements results in \cite[Thm.\ 1]{HZhu}.
Applications (iii) and (iv) build upon \cite[9.7.13]{katzsarnak} and \cite[6.1]{kowalskisieve} respectively. 

\paragraph{Applications:} 
Let $\ff$ be a finite field of characteristic $p$.  Under the
hypotheses of Theorem \ref{Tprankmono}:
\begin{description}
\item{(i)} 
there is an $\overline{\ff}$-curve $C$ of genus $g$ and $p$-rank $f$ with $\aut_{\bar\ff}(C) = \st{\id}$ (\ref{apptrivaut});
\item{(ii)} 
there is an $\overline{\ff}$-curve $C$ of genus $g$ and $p$-rank $f$ whose Jacobian is absolutely simple (\ref{appabssimp});
\item{(iii)}  
if $\abs{\ff}\equiv 1 \bmod \ell$, about $\ell/(\ell^2-1)$ of the $\ff$-curves of genus $g$ and $p$-rank $f$ 
have a point of order $\ell$ on their Jacobian (\ref{appclass});
\item{(iv)} 
for most $\ff$-curves $C$ of genus $g$ and $p$-rank $f$, the
  splitting field of the numerator of the zeta function of $C$ has degree $2^g g!$ over $\rat$ (\ref{appzeta}).
\end{description}

At its heart, this paper relies on fundamental work of Chai and Oort.
The proof of Theorem \ref{Tprankmono} appears in Section
\ref{Smonoresults}.  It proceeds by degeneration (as in
\cite{ekedahlmono}) and induction on the genus.  Consider
the moduli space $\cala_g$ of principally polarized abelian varieties of
dimension $g$ and its $p$-rank strata $\cala_g^f$.  Recent work in
\cite{chailadic} gives information about the integral monodromy of
$\cala_g^f$.  In particular, an irreducible subspace of
$\cala_g$ which is stable under all Hecke correspondences and whose
generic point is not supersingular has monodromy group
$\sp_{2g}(\integ/\ell)$.  The base cases of Theorem \ref{Tprankmono}
rely on the fact that the dimensions of $\calm_g^f$ and $\cala_g^f$ are
equal if $g \leq 3$.

We note that \cite{chailadic} is not directly applicable to the strata
$\calm_g^f$ when $g \geq 4$.  When $g \geq 4$, the Torelli locus is
very far from being Hecke-stable.  Another method for computing
monodromy groups is found in \cite{hall06}, where the author shows
that certain group-theoretic conditions on the local inertia structure
of a $\integ/\ell$-sheaf guarantee that its global monodromy group is
the full symplectic group.  The method of \cite{hall06}
applies only to families of curves in which the Jacobian of at least
one degenerate fiber has a nontrivial toric part, and thus does not apply to $\calm_g^0$.

The inductive step of Theorem \ref{Tprankmono}
uses results about the boundary of $\calm_g^f$
found in Section \ref{Sboundresults}.  In particular, it employs a new
result that the closure of every component $S$ of $\calm_g^f$ in $\bar
\calm_g$ contains moduli points of chains of curves of specified
genera and $p$-rank (Proposition \ref{propchain}), and in particular
intersects the boundary component $\Delta_{1,1}$ in a certain way
(Corollary \ref{cordegen}).  As in \cite{achterpries07}, this
implies that the monodromy group of $S$ contains two non-identical
copies of $\sp_{2g-2}(\integ/\ell)$, and is thus isomorphic to
$\sp_{2g}(\integ/\ell)$.
  
A result of independent interest in Section \ref{Sboundresults} is the following.  

\begin{corboundary}
Suppose $g \ge 2$ and $ 0 \le f \le g$.  Let $\Omega \subset \st{1,
\ldots, g}$ be a subset of cardinality $f$.  Let $S$ be an irreducible
component of $\calm_g^f$.  Then $\bar S$ contains the moduli point of a
chain of elliptic curves $E_1$, $\ldots$, $E_g$, where $E_j$ is
ordinary if and only if $j \in \Omega$.
\end{corboundary}

In Section \ref{Squestions}, we include some open questions about the
geometry of the $p$-rank strata of curves.  For example, the number of
irreducible components of $\calm_g^f$ is known only in special cases.
Finally, we anticipate that the techniques of this paper can be used
to compute the $\ell$-adic monodromy of components of the $p$-rank
strata $\calh_g^f$ of the moduli space $\calh_g$ of hyperelliptic
curves of genus $g$ as well.

We thank the referee for helpful comments.

\section{Background} \label{S2}

Let $k$ be an algebraically closed field of characteristic $p>0$.  In
Sections \ref{S2}, \ref{Sboundresults}, and \ref{Smono} all objects
are defined on the category of $k$-schemes, and $T$ is an arbitrary
$k$-scheme.  Let $\ell$ be a prime distinct from $p$.  We fix an
isomorphism $\mmu_\ell \simeq \integ/\ell$.

\subsection{Moduli spaces}
\label{subsecmoduli}

For each $g \ge 1$ consider the following well-known categories, each of which is fibered
in groupoids over the category of $k$-schemes in its \'etale topology:
\begin{description}
\item{$\cala_g$} principally polarized abelian
  schemes of dimension $g$;
\item{$\calm_g$} smooth connected proper relative
  curves of genus $g$;
\item{$\bar\calm_g$} stable relative curves of genus $g$.
\end{description}
For each positive integer $r$, there is also (see \cite[Def.\ 1.1,1.2]{knudsen2}) the category
\begin{description}
\item{$\bar\calm_{g;r}$} $r$-labeled stable relative curves $(C; P_1,
\ldots, P_r)$ of genus $g$.
\end{description}
These are all smooth Deligne-Mumford stacks, and $\bar\calm_g$ and
$\bar\calm_{g;r}$ are proper \cite[Thm.\ 2.7]{knudsen2}.  
There is a forgetful functor $\phi_{g;r}:
\bar\calm_{g;r} \ra \bar\calm_g$.  Let $\bar \calm_{g,0} =
\bar\calm_g$.  Let $\calm_{g;r} =
\bar\calm_{g;r}\cross_{\bar\calm_g} \calm_g$ be the moduli stack of
$r$-labeled smooth curves of genus $g$.  The boundaries of
$\bar\calm_g$ and $\bar\calm_{g;r}$ are $\del\bar\calm_g = \bar\calm_g
- \calm_g$ and $\del\bar\calm_{g;r} = \bar\calm_{g;r} - \calm_{g;r}$,
respectively.   If $S\subset \bar\calm_g$, let $\bar S$ be the closure of $S$ in
$\bar\calm_g$.

For a $k$-scheme $T$, $\calm_g(T) = {\rm Mor}_k(T,\calm_g)$ is the 
category of smooth proper relative curves of genus $g$ over $T$.
There is a tautological curve $\calc_g$ over the moduli stack
$\calm_g$ \cite[Sec. 5]{delignemumford}.
If $s \in \calm_g(k)$, let $\calc_{g,s}$ denote the fiber of $\calc_g$ over $s$, which 
is the curve corresponding to the point $s: \spec k \ra \calm_g$.  
Similar conventions are employed for the tautological marked curve $\calc_{g;r}$ over $\bar \calm_{g;r}$.

Let $C/k$ be a stable curve.  The Picard variety $\pic^0(C)$ is an
abelian variety if each irreducible component of $C$ is smooth and if
the intersection graph of the irreducible components of $C$ is a tree.
Such a curve is said to be of compact type.  Curves which are not
of compact type correspond to points of a component $\Delta_0$ (defined in
Section \ref{Sclutch}) of $\del\bar\calm_g$.

\subsection{The $p$-rank}

Let $X$ be a principally polarized abelian variety of dimension $g$
over an algebraically closed field $k'$ of characteristic $p$.  The
$p$-rank of $X$ is the integer $f$ such that $X[p](k') \iso
(\integ/p)^f$. It may be computed as $f(X) =
\dim_{\ff_p}\hom(\mmu_p,X)$, where $\mmu_p$ is the kernel of Frobenius
on the multiplicative group $\gp_m$.  It is well-known that $0 \leq f
\leq g$.  This definition extends to semiabelian varieties; if $X/k'$
is a semiabelian variety, its $p$-rank is
$\dim_{\ff_p}\hom(\mmu_p,X)$.  If $X$ is an extension of an abelian
variety $Y$ by a torus $W$, then $f(X) = f(Y) + \dim(W)$.  If $X/k_0$
is a semiabelian variety over an arbitrary field of characteristic
$p$, its $p$-rank is that of $X_{k'}$ for any algebraically closed
field $k'$ containing $k_0$.
If $C/k'$ is a stable curve, then its $p$-rank $f(C)$ is that of $\pic^0(C)$.

\begin{lemma}
  Let $X \ra S$ be a semiabelian scheme of relative dimension $g$ over
  a Deligne-Mumford stack, and suppose $0 \le f \le g$.  There is a
  locally closed reduced substack $S^f$ of $S$ such that for each
  field $k'\supset k$ and point $s \in S(k')$, then $s \in S^f(k')$ if
  and only if the $p$-rank of $X_s$ is $f$.
\end{lemma}

\begin{proof}
  A substack of $S$ is reduced and locally closed if it is locally representable
  by reduced locally closed subschemes \cite[p.\ 100]{delignemumford}
  \cite[3.9 and 3.14]{lmbstacks}.  Therefore, it suffices to consider
  the case that $S$ is an affine scheme.  Write $X$ as an extension $0
  \ra W \ra X \ra Y \ra 0$, where $Y$ is an abelian scheme and $W$ is
  a torus.  Since $\dim(W)$ is an upper semicontinuous function on the
  base \cite[p.\ 8]{faltingschai}, there is a finite stratification of
  $S$ by locally closed subschemes on which $\dim(W)$ is constant.
  Since a finite union of locally closed subschemes is again locally
  closed, one may assume that $\dim(W)$ is constant.  Finally, since
  $f(Y) = f(X)+\dim(W)$, it suffices to prove the result for the
  abelian scheme $Y$.  The existence of $S^f$ then follows immediately
  from \cite[Thm.\ 3.2.1]{katzsf}.
\end{proof}

In particular, $\cala_g^f$ and  $\calm_g^f$ denote the
locally closed reduced substacks of $\cala_g$ and $\calm_g$,
respectively, 
whose geometric points correspond to objects with $p$-rank $f$.  
Similary, $\bar\calm_g^f:=(\bar \calm_g)^f$ and $\bar \calm_{g;r}^f:=(\bar \calm_{g:r})^f$.
Note that $\bar \calm_g^f$ may be strictly contained in $\bar{\calm_g^f}$ since the 
latter may contain points $s$ such that $f(\calc_{g,s})<f$.

Every component of $\bar\calm_g^f$ has dimension $2g-3+f$
\cite[Thm.\ 2.3]{FVdG:complete}.
Since $\bar \calm_{g;r}^f$ is the fibre of
$\phi_{g,r}$ over $\bar \calm_{g;r}^f$, it is pure of dimension $2g-3+f+r$.

\subsection{Clutching maps} \label{Sclutch}

If $g_1, g_2, r_1, r_2$ are positive integers, there is 
a clutching map

\[
\xymatrix{
\kappa_{g_1;r_1, g_2;r_2}:\bar\calm_{g_1;r_1}\cross\bar\calm_{g_2;r_2} \ar[r] &\bar\calm_{g_1+g_2;r_1+r_2-2}.
}
\]
Suppose $s_1 \in \bar\calm_{g_1;r_1}(T)$ is the moduli point of the
labeled curve $(C_1;P_1, \ldots, P_r)$, and suppose $s_2 \in
\bar\calm_{g_2;r_2}(T)$ is the moduli point of $(C_2;Q_1, \ldots,
Q_{r_2})$.  Then $\kappa_{g_1;r_1,g_2;r_2}(s_1,s_2)$ is the moduli
point of the labeled $T$-curve $(D; P_1, \ldots, P_{r_1-1},
Q_2, \ldots Q_{r_2})$, where the underlying curve $D$ has 
components 
$C_1$ and $C_2$, the sections $P_{r_1}$ and $Q_1$ are identified in an
ordinary double point, and this nodal section is dropped from the
labeling.  The clutching map is a closed immersion if $g_1\not = g_2$
or if $r_1+r_2
\ge 3$, and is always a finite, unramified map \cite[Cor.\ 3.9]{knudsen2}.

By \cite[Ex.\ 9.2.8]{blr},
\begin{align}
\label{eqblr}
\pic^0(D) &\iso \pic^0(C_1) \cross \pic^0(C_2).
\intertext{Then the $p$-rank of $E$ is}
\label{eqblrprank}
f(D) &= f(C_1)+f(C_2).
\end{align}

Similarly, if $g$ is a positive integer and if $r \ge 2$, there is a
map
\[
\xymatrix{
\kappa_{g;r}:\bar\calm_{g;r} \ar[r] &\bar\calm_{g+1;r-2}.
}
\]
If $s \in \bar \calm_{g;r}(T)$ is the moduli point of the labeled curve 
$(C; P_1, \ldots, P_r)$ then $\kappa_{g;r}(s)$ is the moduli point of the
labeled curve $(E; P_1, \ldots, P_{r-2})$ where $E$ is obtained by identifying 
the sections $P_{r-1}$ and $P_r$ in an ordinary double point, and these 
sections are subsequently dropped from the labeling. 
Again, the 
morphism $\kappa_{g;r}$ is finite and unramified \cite[Cor.\
3.9]{knudsen2}.

By \cite[Ex.\ 9.2.8]{blr}, $\pic^0(E)$ is an extension
\begin{equation}
\label{eqblr0}
\xymatrix{
0 \ar[r] & W \ar[r] & \pic^0(E)  \ar[r] & \pic^0(C) \ar[r] & 0},
\end{equation}
where $W$ is a one-dimensional torus.  
In particular, the toric rank of $\pic^0(E)$ is one greater than that of $\pic^0(C)$, 
and their maximal projective quotients are isomorphic, so that 
\begin{equation}
\label{eqblrprank0}
f(E) = f(C)+1.
\end{equation}
  
For $1 \le i \le g-1$, let $\Delta_i = \Delta_{i}[\bar\calm_g]$ be the
image of $\kappa_{i,1;g-i,1}$.  Note that $\Delta_i$ and $\Delta_{g-i}$ are the
same substack of $\bar\calm_g$.
Let $\Delta_0 = \Delta_0[\bar\calm_g]$ be the image
of $\bar \calm_{g-1;2}$ under $\kappa_{g-1;2}$. Each $\Delta_i$ is an
irreducible divisor in $\bar\calm_g$, and $\del\bar\calm_g$ is the union
of the $\Delta_i$ for $0 \le i \le \floor{g/2}$
 (e.g., \cite[p.190]{knudsen2}).  If $S$
is a stack equipped with a map $S \ra \bar\calm_g$, let
$\Delta_i[S]$ denote $S\cross_{\bar\calm_g} \Delta_i[\bar\calm_g]$. 
Also define $\Delta_i[\bar\calm_g]^f:=(\Delta_i[\bar\calm_g])^f$.

If $g\ge 3$, then there exists a commutative diagram of clutching maps

\begin{equation}
\label{diagclutchbox}
\xymatrix{
\bar\calm_{1,1} \cross \bar \calm_{g-2,2} \cross \bar \calm_{1,1} 
\ar[d] \ar[r] \ar[dr]^{\kappa_{1,g-2,1}}
 & \bar\calm_{g-1,1} \cross \bar \calm_{1,1} \ar[d]
\\
\bar\calm_{1,1}\cross \bar\calm_{g-1,1} \ar[r] & \bar\calm_g.
}
\end{equation}

Let $\Delta_{1,1} = \Delta_{1,1}[\bar\calm_g]$ 
denote the image in $\bar\calm_g$ of the upper left-hand object;
it is the (reduced) self-intersection locus of $\Delta_1$.
There is an open, dense
substack $U_{1,1} \subset \Delta_{1,1}$ such that if $s \in U_{1,1}(k)$,
then $\calc_{g,s}$ is a chain of three irreducible smooth curves $Y_1$, $Y_2$, $Y_3$
with $g_{Y_1}=g_{Y_3}=1$ and $g_{Y_2}=g-2$.   
Also, for $i \in \st{1,3}$, the curves $Y_i$ and $Y_2$ intersect in a
point $P_i$ which is an ordinary double point.

\section{The $p$-rank strata of curves}
\label{Sboundresults} 

\subsection{Boundary of the $p$-rank strata of curves}
 
The $p$-rank strata of the boundary of $\bar \calm_g$ are easy to
describe using the clutching maps.  First, if $f \geq 1$, then
$\Delta_0[\bar \calm_g]^f$ is the image of $\bar \calm_{g-1;2}^{f-1}$
under $\kappa_{g-1;2}$ by \eqref{eqblrprank0}.  Second, if $1 \leq i \leq
 g-1$ and $0 \leq f \leq g$, then \eqref{eqblrprank}
implies that $\Delta_i[\bar \calm_g]^f$ is the union of the images of
$\bar\calm_{i;1}^{f_1} \cross \bar \calm_{g-i;1}^{f_2}$ under
$\kappa_{i;1, g-i;1}$ as $(f_1,f_2)$ ranges over all pairs such that
\begin{equation}
\label{f1f2conditions}
0 \leq f_1 \leq i,\ 0 \leq f_2 \leq g-i \text{ and } f_1+f_2=f.
\end{equation}
 
\begin{lemma}
\label{lemdimw}  Suppose $g \geq 2$ and $0 \le f \le g$.
If $0 \le i \le g-1$  and $(f,i) \not =(0,0)$, then every component 
of $\Delta_i[\bar\calm_g]^f$ has dimension $2g+f-4$.
\end{lemma}

\begin{proof}
Suppose $1 \le f \le g$.  Then $\bar\calm_{g-1;2}^{f-1}$
is pure of dimension $\dim(\bar \calm_{g-1}^{f-1})+2 =2g+f-4$.  Since
$\kappa_{g-1;2}$ is finite, $\Delta_0[\bar\calm_g]^f$ is pure of
dimension $2g+f-4$ as well. 

Similarly, suppose $0 \le f \le g$ and $1 \le i \le g-1$.  Let
$(f_1,f_2)$ be any pair of integers satisfying \eqref{f1f2conditions}.
Then $\bar\calm_{i;1}^{f_1} \cross \bar \calm_{g-i;1}^{f_2}$ is pure of dimension
$\dim(\bar \calm_{i;1}^{f_1}) + \dim(\bar \calm_{g-i;1}^{f_2}) = 2g+f -4$.  Since
$\kappa_{i;1,g-i;1}$ is finite, $\Delta_i[\bar\calm_g]^f$ is pure of dimension $2g+f-4$ as well.
\end{proof}

The first part of the next lemma shows that 
if $\eta$ is a generic point of $\bar\calm_g^f$, then the
curve $\calc_{g,\eta}$ is smooth. 
Thus no component of $\bar\calm_g^f$ is
contained in the boundary $\del\bar\calm_g$.   The last part shows
that one can adjust the labeling of an $r$-labeled curve of genus $g$ and
$p$-rank $f$ without leaving the irreducible component of
$\bar\calm_{g;r}^f$ to which its moduli point belongs.

\begin{lemma}
\label{lemlabeled}
Suppose $g \ge 1$, $0 \le f \le g$, and $r \ge 1$.
\begin{alphabetize}
\item Then $\calm_g^f$ is open and dense in $\bar\calm_g^f$.
\item Then $\calm_{g;r}^f$ is open and dense in $\bar\calm_{g;r}^f$.
\item Let $S$ be an irreducible component of $\bar\calm_{g;r}^f$.
Then $S = \phi_{g;r}\inv(\phi_{g;r}(S))$.  Equivalently, if $T$ is a
$k$-scheme, if $(C; P_1, \ldots, P_r) \in S(T)$, and if $(Q_1, \ldots, Q_r)$ is
any other labeling of $C$, then $(C; Q_1, \ldots, Q_r) \in S(T)$.
\end{alphabetize}
\end{lemma}

\begin{proof}
Part (a) is well-known if $g = 1$.  For $g \ge 2$, the result follows
immediately from Lemma \ref{lemdimw}, since $\bar\calm_g^f$ is pure of
dimension $2g+f-3$ \cite[Thm.\ 2.3]{FVdG:complete}.
Part (b) follows from the fact that the $p$-rank of a labeled curve
depends only on the underlying curve, so that $\bar\calm_{g;r}^f =
\bar\calm_{g;r} \cross_{\bar \calm_g} \bar\calm_g^f$.

For part (c), let $S$ be an irreducible component of
$\bar\calm_{g;r}^f$.  It suffices to show that
$\phi_{g;r}\inv(\phi_{g;r}(S)) \subseteq S$.
By part (b), $U = S \cap \calm_{g;r}$ is open and dense in $S$.  Therefore,
$S$ is the largest irreducible substack of $\bar\calm_{g;r}^f$ which
contains $U$. The fibers of $\phi_{g;r}\rest{U}$ are irreducible, so
$\phi_{g;r}\inv(\phi_{g;r}(U))$ is also an irreducible substack of $\bar 
\calm_{g;r}^f$ which contains $U$.  Thus $\phi_{g;r}\inv(\phi_{g;r}(U))\subset
S$. This shows that $\phi_{g;r}\inv(\phi_{g;r}(U)) = U$.  

To finish the proof, it suffices to show that a $T$-point of $S$
is a $T$-point of $\phi_{g;r}\inv(\phi_{g;r}(S))$ for an arbitrary $k$-scheme
$T$.  To this end, let $\alpha = (C; P_1, \ldots, P_r) \in S(T)$,
and let $\beta = (C; Q_1, \ldots, Q_r) \in \bar\calm_{g;r}^f(T)$.
Note that $\phi_{g;r}(\beta) = \phi_{g;r}(\alpha)$, and
$\phi_{g;r}(\alpha)$ is supported in the closure of $\phi_{g;r}(U)$ in
$\bar\calm_g^f$.  Because $\calm_{g;r}$ is dense in $\bar\calm_{g;r}$,
it follows that $\beta$ is supported in the closure of $\phi_{g;r}\inv(\phi_{g;r}(U))$ in
$\bar\calm_{g;r}^f$, which is $S$. 
\end{proof}

\begin{lemma} \label{LinterDelta0}
Suppose $g \ge 2$ and $0 \le f \le g$.  Let $S$ be an irreducible component of $\calm_g^f$.
\begin{alphabetize}
\item Then $\bar S$ intersects $\Delta_0[\bar \calm_g]$ if and only if $f \ge 1$.
\item If $f \ge 1$, then each irreducible component of $\Delta_0[\bar S]$ contains
the image of a component of $\bar\calm_{g-1,2}^{f-1}$ under $\kappa_{g-1;2}$.
\end{alphabetize}
\end{lemma}

In other words, Lemma \ref{LinterDelta0}(b) states that if $f \geq 1$ and 
if $\eta$ is a generic point of $\Delta_0[\bar S]$, 
then the normalization $\til \calc_{g,\eta}$ of $\calc_{g,\eta}$ is a smooth
curve of genus $g-1$ and $p$-rank $f-1$. 
The containment in Lemma \ref{LinterDelta0}(b) may be strict since
$\Delta_0[\bar S]$ may contain points $s$ such that $\calc_{g,s}$ has $p$-rank strictly less than $f$.

\begin{proof}
If $f = 0$, 
then equation \eqref{eqblrprank0} implies that $\bar S$ does not intersect $\Delta_0[\bar \calm_g]$.
If $f \geq 1$, then $\bar S \subset \bar \calm_g$ is a complete
substack of dimension greater than $2g-3$.  Let $\caln_g = \bar\calm_g
- \Delta_0$ be the sublocus of curves of compact type; it is open in
$\bar \calm_g$.  By \cite[Lemma 2.4]{FVdG:complete}, a complete
substack of $\caln_g$ has dimension at most $2g-3$.  Thus, $\bar S$ is
not contained in $\caln_g$ and so $\bar S$ intersects $\Delta_0$
nontrivially.  This completes part (a).

For part (b), recall that $\bar S$ is not contained in $\del
\bar\calm_g$ by Lemma \ref{lemlabeled}(a). 
Since $\bar\calm_g$, $\bar S$ and $\del \bar \calm_g$ are proper, 
the intersection of $\bar S$ with the divisor $\Delta_0[\bar\calm_g]$
has pure dimension $\dim\bar S - 1$, which equals $\dim (\Delta_0[\bar\calm_g]^f)$
by Lemma \ref{lemdimw}.  Thus each irreducible component
of $\bar S \cap \Delta_0[\bar\calm_g]$ contains the image of some
component of $\bar \calm_{g-1;2}^{f-1}$ under the finite morphism $\kappa_{g-1;2}$.
In particular, it contains the image of the moduli point of a smooth curve of genus
$g-1$ and $p$-rank $f-1$.
\end{proof}

The next result shows that the closure of each irreducible component of $\calm_g^f$ 
intersects $\Delta_i[\bar \calm_g]^f$ in every way possible.

\begin{proposition} \label{PinterDelta1}
Suppose $g \ge 2$ and $0 \le f \le g$.  
Suppose $1 \le i \le g-1$ and 
$(f_1,f_2)$ is a pair satisfying the conditions in \eqref{f1f2conditions}.
Let $S$ be an irreducible component of $\calm_g^f$.
\begin{alphabetize}
\item 
Then $\bar S$ intersects $\kappa_{i; 1, g-i;1}(\bar \calm_{i;1}^{f_1} \times \bar \calm_{g-i;1}^{f_2})$.
\item Each irreducible component of the intersection contains the image of a component of 
$\bar \calm_{i;1}^{f_1} \times \bar \calm_{g-i;1}^{f_2}$.
\end{alphabetize}
\end{proposition}

Proposition \ref{PinterDelta1}(b) implies that if $\eta$ is a generic point
of $\bar S \cap \kappa_{i; 1, g-i;1}(\bar
\calm_{i;1}^{f_1} \times \bar \calm_{g-i;1}^{f_2})$, then
$\calc_{g,\eta}$ is a chain of two smooth curves of respective genera
$i$ and $g-i$ and respective $p$-ranks $f_1$ and $f_2$.

\begin{proof}
{\bf (a) implies (b):}
To see that (a) implies (b) for fixed $g$, $f$, $i$, and $(f_1,f_2)$, note that $\bar S \cap
  \Delta_i[\bar \calm_g]^f$ has pure dimension $\dim \bar S - 1 = 2g+f-4$
by the same reasoning as in the proof of Lemma \ref{LinterDelta0}(b).
  By Lemma \ref{lemdimw} this is the same as the dimension of each
  component of $\kappa_{i;1, g-i;1}(\bar \calm_{i;1}^{f_1} \times \bar
  \calm_{g-i;1}^{f_2})$. 

\bigskip

{\bf Base cases:}
The proof of (a) is by induction on $g$ while holding $g-f$ fixed.
The two base cases are when $f = 0$ and when $g=2$.
When $f=0$, then the only possibility for $(f_1,f_2)$ is $(0,0)$.  By
\ref{LinterDelta0}(a), $\bar S$ is contained in $\caln_g =\bar \calm_g
- \Delta_0$.  By \cite[Lemma 2.5]{FVdG:complete}, $\bar S$ intersects
$\Delta_i[\bar \calm_g]$, and a point in the intersection must be in
$\kappa_{i; 1, g-i;1}(\bar \calm_{i;1}^{0} \times \bar
\calm_{g-i;1}^{0})$.

For the other base case, let $g=2$.  The statement is true for $g=f=2$
(or more generally when $g=f$) because $\calm_g^g$ is open and dense
in $\bar \calm_g$.  If $f=1$, then $i=1$.  Without loss of generality, $(f_1, f_2)=(1,0)$.  
Thus the next claim suffices to conclude the proof of the base cases.

\bigskip

{\bf Claim:}
$\bar \calm_2^1$ is irreducible and intersects $\kappa_{1; 1, 1;1}(\bar
\calm_{1;1}^{1} \times \bar \calm_{1;1}^{0})$.

To prove the claim, recall that the Torelli morphism $\calm_2 \ra \cala_2$
is an inclusion \cite[Lemma 1.11]{oortsteenbrink}.
Since $\dim (\calm_2^1) =\dim(\cala_2^1)$, and since  $\cala_2^1$ is irreducible (e.g., \cite[Ex.\
11.6]{evdg}), it follows that $\calm_2^1$ is irreducible.  Consider a
chain of two smooth elliptic curves, one of which is ordinary and one
supersingular, intersecting in an ordinary double
point.  The moduli point of this curve is in $\bar \calm_2^1 \cap\kappa_{1; 1, 1;1}(\bar
\calm_{1;1}^{1} \times \bar \calm_{1;1}^{0})$.

\bigskip

{\bf Inductive step:} Suppose that $g \geq 3$ and $1 \le f \le g$.  The inductive
hypothesis is that, given $1 \le i' \le g-2$, given a pair
$(f'_1,f'_2)$ such that $0 \leq f'_1 \leq i'$, $0 \leq f'_2 \leq
g-1-i'$, and $f'_1+f'_2 =f-1$, and given an irreducible component $Z$
of $\calm_{g-1}^{f-1}$, then $\bar Z$ intersects $\kappa_{i'; 1,
g-1-i';1}(\bar \calm_{i';1}^{f'_1} \times \bar
\calm_{g-1-i';1}^{f'_2})$; and, since (a) implies (b),
that each irreducible component of the intersection contains a component of $\kappa_{i'; 1,
g-1-i';1}(\bar \calm_{i';1}^{f'_1} \times \bar \calm_{g-1-i';1}^{f'_2})$.

Let $S$ be an irreducible component of $\calm_g^f$.  Let $1 \le i \le
g-1$, and let $(f_1,f_2)$ be a pair which satisfies
\eqref{f1f2conditions}.  Possibly after exchanging $i$ with $g-i$ and
$f_1$ with $f_2$, we suppose that $f_1 > 0$.  

\bigskip

{\bf Case 1:} Suppose $i>1$.  Let $(g_1',g_2') = (i-1,g-i)$, and let $(f_1',f_2') = (f_1-1,f_2)$.
Note that $g_1'+g_2' = g-1$ and $f_1'+f_2'= f-1$.
Also $0 \le f_j' \le g_j'$ and $g'_j \ge 1$ for $j=1,2$.
So $1 \leq g_1' \leq g-2$.
By Lemma \ref{LinterDelta0}(b), $\bar S$ contains the image of a
component $\til Z$ of $\bar\calm_{g-1;2}^{f-1}$ under
$\kappa_{g-1;2}$.  The inductive hypothesis, applied to (the interior
of) the component $\bar Z =\phi_{g-1;2}(\til Z)$ of
$\bar\calm_{g-1}^{f-1}$, shows that $\bar Z$ intersects
$\kappa_{g_1';1,g_2';1}(\bar \calm_{g_1';1}^{f_1'} \cross \bar
\calm_{g_2';1}^{f_2'})$.  

Let $\xi$ be a generic geometric point of the intersection.  By part
(b) of the inductive hypothesis, the curve $\calc_{g-1,\xi}$ is a chain of two
irreducible curves, $Y_1$ and $Y_2$, with respective genera $g_1'$
and $g_2'$ and $p$-ranks $f_1'$ and $f_2'$, intersecting in one point
$P$ which is an ordinary double point.  Let $P_1$ and $P_2$ be two
distinct points of $Y_1 - \st P$.  Note that after identifying the
points $P_1$ and $P_2$ on $Y_1$, one obtains a (singular) curve of
genus $g_1'+1=i$ with $p$-rank $f_1'+1=f_1$.
The moduli point $\til \xi$ of the labeled curve
$(\calc_{g-1,\xi},\st{P_1,P_2})$ lies in $\til Z$ (Lemma \ref{lemlabeled}(c)).  Then $\kappa_{g-1;2}(\til \xi) \in
\bar S \cap \kappa_{i;1,g-i;1}(\bar\calm_{i;1}^{f_1}\cross \bar
\calm_{g-i;1}^{f_2})$.  

\bigskip

{\bf Case 2:} Suppose $i=1$.  Then $g-i = g-1$  and $f_1 = 1$.  If $f_2
= f-1 > 0$, then case (1) applies after reindexing.  Therefore, it
suffices to consider the remaining case $(i,g-i) = (1,g-1)$ and
$(f_1,f_2) = (1,0)$.
Let $S$ be an irreducible component of $\calm_g^1$.  Then
$\bar S$ contains an irreducible component of
$\kappa_{2;1,g-2;1}(\bar\calm_{2;1}^1 \cross \bar\calm_{g-2;1}^0)$ by case (1).  
By the claim above, and the implication (a) $\to$ (b),  
$\bar\calm_2^1$ contains the image of an irreducible component of
$\bar \calm_{1;1}^1 \cross \bar \calm_{1;1}^0$ under $\kappa_{1;1,1;1}$.
Thus $\bar \calm_{2;1}^1$ contains the image of an irreducible
component of $\bar \calm_{1;1}^1\cross \bar \calm_{1;2}^0$ under $\kappa_{1;1,1;2}$.
Therefore, $\bar S$ contains the image of an irreducible component of
$\bar \calm_{1;1}^1 \cross \bar \calm_{1;2}^0 \cross \bar
\calm_{g-2;1}^0$ under $\kappa_{1,1,g-2}$.  In particular, $\bar S$ has
nonempty intersection with $\kappa_{1;1,g-1;1}(\bar \calm_{1;1}^1 \cross \bar
\calm_{g-1;1}^0)$.
\end{proof}

The next result shows that components of $\bar \calm_g^f$ deeply
intersect the boundary of $\bar \calm_g$. 

\begin{proposition}
\label{propchain}
Suppose $g \ge 2$ and $0 \le f \le g$.  Let $S$ be an irreducible
component of $\calm_g^f$.  Suppose $g_1, \ldots, g_m$ are positive
integers and $f_1, \ldots, f_m$ are integers such that $\sum g_i = g$,
$\sum f_i = f$, and $0 \le f_i \le g_i$ for each $1 \le i \le m$.
Then $\bar S$ contains the image of a component of 
\[
\bar \calm_{g_1;1}^{f_1} \cross \bar \calm_{g_2;2}^{f_2} \cross \bar
\calm_{g_3;2}^{f_3} \cross \cdots \cross \bar
\calm_{g_{m-1};2}^{f_{m-1}}\cross \bar \calm_{g_m;1}^{f_m}
\]
under the clutching map.
\end{proposition}

\begin{proof}
The proof is by induction on $m$, and is proved for all $g$
simultaneously.  The case $m=1$ is trivial, while the case $m=2$ is
Proposition \ref{PinterDelta1}.

Now suppose the result is known for every genus $g'$ and every
partition of $g'$ with at most $m-1$ parts.  Suppose $g$, $f$, $g_1,
\ldots, g_m$, 
and $f_1, \ldots, f_m$ satisfy the conditions of Proposition \ref{propchain}.
Let $g' = \sum_{i=1}^{m-1} g_i$ and $f' = \sum_{i=1}^{m-1}f_i$.  By
Proposition \ref{PinterDelta1}, there exist components $Z'$ of $\bar
\calm_{g';1}^{f'}$ and  $Z$ of $\bar \calm_{g_m;1}^{f_m}$
such that $\bar S$ contains $\kappa_{g';1,g_m;1}(Z'\cross Z)$.  By the
inductive hypothesis, $\phi_{g';1}(Z')$ contains $\kappa(Y)$ for some 
component $Y$ of
\[
\bar \calm_{g_1;1}^{f_1} \cross \bar \calm_{g_2;2}^{f_2} \cross \bar
\calm_{g_3;2}^{f_3} \cross \cdots \cross \bar
\calm_{g_{m-2};2}^{f_{m-2}}\cross \bar \calm_{g_{m-1};1}^{f_{m-1}}.
\]
Thus, $Z'$ contains $\kappa(Y)\cross_{\bar\calm_{g'}}
\bar\calm_{g';1}$.  
By Lemma \ref{lemlabeled}(c), one can choose the
labeled point to be supported on the last component of $\calc_{g',\kappa(Y)}$.
Thus, $Z'$ contains the image of a component $X$ of
\[
\bar \calm_{g_1;1}^{f_1} \cross \bar \calm_{g_2;2}^{f_2} \cross \bar
\calm_{g_3;2}^{f_3} \cross \cdots \cross \bar
\calm_{g_{m-2};2}^{f_{m-2}}\cross \bar \calm_{g_{m-1};2}^{f_{m-1}}.
\]
Then
$\bar S$ contains the image of $X\cross Z$.
\end{proof}

The next corollary is not used in the rest of the paper, but is
included to show that Proposition \ref{propchain} generalizes
\cite[Lemma 2.5]{FVdG:complete}, which states that every component of
$\bar\calm_g^0$ contains the moduli point of a chain of supersingular
elliptic curves.

\begin{corollary} \label{Cchainelliptic}
Suppose $g \ge 2$ and $ 0 \le f \le g$.  Let $\Omega \subset \st{1,
\ldots, g}$ be a subset of cardinality $f$.  Let $S$ be an irreducible
component of $\calm_g^f$.  Then $\bar S$ contains the moduli point of a
chain of elliptic curves $E_1$, $\ldots$, $E_g$, where $E_j$ is
ordinary if and only if $j \in \Omega$.
\end{corollary}

\begin{proof}
This follows immediately from Proposition \ref{propchain}.
\end{proof}

The next result is the form of Proposition \ref{propchain} used to
prove Theorem \ref{Tprankmono}, which relies on degeneration
to $\Delta_{1,1}$.  We label the four possibilities for
$(f_1,f_2,f_3)$ such that $f_1 +f_2+f_3=f$ and $f_1,f_3 \in \{0,1\}$
as follows:
(A) $(1,f-2,1)$;
(B) $(0,f-1,1)$; 
(B') $(1,f-1,0)$; and
(C) $(0,f,0)$.

\begin{corollary} \label{cordegen}
Suppose $g \geq 3$ and $0 \leq f \leq g$.  Let $S$ be an irreducible component of $\calm_g^f$.  
\begin{alphabetize}
\item Then $\bar S$ intersects $\Delta_{1,1}[\bar \calm_g]$.
\item There is a choice of $(f_1,f_2,f_3)$ from cases (A)-(C), 
and there are irreducible components $S_1$ of $\bar\calm_{1;1}^{f_1}$
and $S_2$ of $\bar\calm_{g-2;2}^{f_2}$ 
and $S_3$ of $\bar \calm_{1;1}^{f_3}$;
and there are irreducible components $S_R$ of $\bar\calm_{g-1;1}^{f_2+f_3}$ and $S_L$ of 
$\bar\calm_{g-1;1}^{f_1+f_2}$;
so that the restriction of the clutching morphisms gives a commutative diagram
\begin{equation}
\label{diagclutch}
\xymatrix{
S_1 \cross S_2 \cross S_3 \ar[r]  \ar[d]& 
 S_1 \cross S_R \ar[d]\\
 S_L \cross S_3 \ar[r] & \bar S \cap \Delta_{1,1}
}
\end{equation}
\item Furthermore, case (A) occurs as long as $f \geq 2$, case (B) or (B') occurs as long as $1 \leq f \leq g-1$, 
and case (C) occurs as long as $f \leq g-2$. 
\end{alphabetize}
\end{corollary}

\begin{proof}
All three parts follow immediately from Proposition \ref{propchain}.
\end{proof}

\subsection{Open questions about the geometry of the $p$-rank strata} \label{Squestions}

The phrasing of Corollary \ref{cordegen} is as simple as possible 
given the present lack of information about the number of components
of $\calm_g^f$.
We note that if $\calm_g^f$ is irreducible for some pair $g,f$,
then there are much shorter proofs of Corollaries \ref{Cchainelliptic} and \ref{cordegen}.

\begin{question} \label{Q1}
How many irreducible components does $\calm^f_g$ have?
\end{question}

The answer to Question \ref{Q1} appears to be known for all $p$ only
in the following cases: $\calm^f_g$ is irreducible when $f=g$;
$\calm^f_g$ is irreducible when $g=2$ and $f \geq 1$ and when $g=3$
and $0 \leq f \leq 3$; for $\calm_1^0$ or $\calm_2^0$, the number of
irreducible components 
is a class number associated with the quaternion algebra ramified at
$p$ and $\infty$.  If $g \geq 3$, then $\cala_g^f$ is irreducible \cite[Remark 4.7]{chailadic}.

\begin{question}
\label{Q2}
Let $S$ be an irreducible component of $\calm_g^f$.  
Does the closure of $S$ in $\calm_g$ contain an irreducible component
of $\calm_g^0$?
\end{question}

The analogous property is true for $\cala_g^f$ by
\cite[Remark 4.7]{chailadic}.  If the answer to Question \ref{Q2} is
affirmative for some $g$, then one can reduce to the case $f = 0$ in the proof of
Theorem \ref{Tprankmono}.

\section{Monodromy} \label{Smono}

\subsection{Definition of monodromy}
\label{subsecmono}

Recall the notations about monodromy found in
\cite[Section 3.1]{achterpries07}.  In particular, let $C \ra S$ be a
stable curve of genus $g$ over a connected $k$-scheme $S$.
There is an open subset $S^\circ\subset S$ such that $\pic^0(C)\rest{S^\circ}$
is an abelian scheme; in fact, $S^\circ = S
\cross_{\bar\calm_g}(\bar\calm_g - \Delta_0)$. Assume that $S^\circ$ is nonempty and connected. 
  Let $s\in S^\circ$ be a
geometric point.

Let $\ell$ be a prime distinct from $p$.  For each positive integer $n$, the
fundamental group $\pi_1(S^\circ,s)$ acts linearly on
$\pic^0(C)[\ell^n]_s$ via a representation $\rho_{C \ra S,s,\ell^n}:
\pi_1(S^\circ,s) \ra \aut(\pic^0(C)[\ell^n]_s)$.  The
$\integ/\ell^n$-monodromy of $C \ra S$, written as 
$\mono_{\ell^n}(C \ra S, s)$, is the image of $\rho_{C\ra
  S,s,\ell^n}$.  If $S$ is equipped with a morphism $f:S \ra \calm$
to one of the moduli stacks $\calm$ defined in Section
\ref{subsecmoduli}, let $\mono_{\ell^n}(S,s)$ denote
$\mono_{\ell^n}(f^*\calc \ra S,s)$, where $\calc \ra \calm$ is the
tautological curve.  The isomorphism class of the abstract group
$\mono_{\ell^n}(S,s)$ is independent of $s$, and we denote it by
$\mono_{\ell^n}(S)$.
Let $\mono_{\integ_\ell}(S) = \invlim n
\mono_{\ell^n}(S)$, and let $\mono_{\rat_\ell}(S)$ be the Zariski
closure of $\mono_{\integ_\ell}(S)$ in $\gl_{2g}(\rat_\ell)$.

A priori, $\mono_{\ell^n}(S)
\subset \gl_{2g}(\integ/\ell^n)$.  The principal polarization
$\lambda$ on $\pic^0(C)$ induces a symplectic pairing
$\ang{\cdot,\cdot}_\lambda$ on the $\ell^n$-torsion, with values in
$\mmu_{\ell^n}$.  Therefore, there is an inclusion $\mono_{\ell^n}(S)
\subseteq \gsp_{2g}(\integ/\ell^n)$ of the monodromy group in the group of
symplectic similitudes.  Moreover, since $k$ contains  ${\ell^n}$th
roots of unity, 
$\mono_{\ell^n}(S) \subseteq \sp_{2g}(\integ/\ell^n)$.  
Similarly, $\mono_\Lambda(S)
\subseteq \sp_{2g}(\Lambda)$ if $\Lambda$ is either $\integ_\ell$ or
$\rat_\ell$.
If the monodromy group is the full symplectic group, this has a geometric
interpretation, as follows.
Equip $(\integ/\ell)^{2g}$ with the standard symplectic pairing
$\ang{\cdot,\cdot}_\std$, and let
\begin{equation*}
S_{[\ell]} := \Isom( (\pic^0(C/S)[\ell], \ang{\cdot,\cdot}_\lambda),
((\integ/\ell)_S^{2g},\ang{\cdot,\cdot}_{\std})).
\end{equation*}
For simplicity, assume that $S = S^\circ$, i.e., that all fibers of $C \ra
S$ are of compact type. There is an $\ell$th root of
unity on $S$, so $S_{[\ell]} \ra S$ is an \'etale Galois cover, possibly
disconnected, with covering group $\sp_{2g}(\integ/\ell)$.  The cover
$S_{[\ell]}$ is connected if and only if $\mono_\ell(S) \iso
\sp_{2g}(\integ/\ell)$.

Finally, since the category of \'etale covers of a Deligne-Mumford
stack is a Galois category \cite[Section 4]{noohi}, one can employ the same
formalism to study a relative curve $C$ over a connected stack
$\cals$.  

Let $h: \calt \ra \cals$ be a morphism of connected stacks.  This
morphism is called a fibration \cite[X.1.6]{sga1} \cite[Section
A.4]{noohi} if every choice of base points $t \in \calt$ and $s
\in \cals$ with $h(t) = s$ induces an exact sequence of homotopy
groups
\begin{equation}
\label{diagfibration}
\xymatrix{
\pi_1(\calt_s,t) \ar[r] &
\pi_1(\calt,t) \ar[r]^{h_*} &
\pi_1(\cals,s) \ar[r] &
\pi_0(\calt_s,t) \ar[r] &
\st{1}.
}
\end{equation}
An arbitrary base change of a fibration is a fibration. 
If $h$ is a fibration and if the fiber $\calt_s$ is
connected, then $h_*$ is a surjection of fundamental groups.
Therefore, if $\calc \ra \cals$ is a relative curve, then the
natural inclusion  $\mono_\ell(h^*\calc \ra \calt, t) \ra \mono_\ell(\calc
\ra \cals, s)$ is an isomorphism.

\begin{lemma}
\label{lemfibration}
Let $g$ and $r$ be positive integers.  The morphism $\bar\calm_{g;r} \ra
\bar\calm_g$ is a fibration with connected fibers. 
\end{lemma} 

\begin{proof} 
By \cite[p.\ 165]{knudsen2}, $\bar\calm_{g;r}$ is
represented by $\calc_{g;r-1}$, the tautological (stable, labeled) curve
over $\bar \calm_{g;r-1}$.
The projection $\calc_{g;r-1} \ra \bar\calm_{g;r-1}$ is proper, flat
and surjective with connected fibers.  By induction on $r$, $\phi_{g;r}: \bar \calm_{g;r}
\ra \bar \calm_g$ is proper, flat and
surjective, and $\phi_{g;r}$ is a fibration \cite[Section A.1]{noohi}.
\end{proof}

\begin{lemma}
\label{LLfibration}
Let $g$ and $r$ be positive integers.  Suppose $S\subset
\bar\calm_g$ is a connected substack and $s$ is a geometric 
point of $S$.  Let $S_r = S \cross_{\bar\calm_g} \bar\calm_{g;r}$, and
let $s_r$ be a lift of $s$ to $S_r$.  Then $\mono_\ell(S,s) = \mono_\ell(S_r,s_r)$.
\end{lemma}

\begin{proof}
By Lemma \ref{lemfibration}, $\bar\calm_{g;r} \ra  \bar\calm_g$ is a fibration with
connected geometric fibers, and thus so is $S_r \ra S$.  The result now
follows from the defining property \eqref{diagfibration} of
fibrations. 
\end{proof}

\subsection{Monodromy of the $p$-rank strata} \label{Smonoresults}

This section contains the main result in the paper, Theorem \ref{Tprankmono}, 
which is about the monodromy of the $p$-rank strata of the moduli space of curves.  
The proof proceeds by induction on the genus of $g$.  
The next lemma, due to Chai, will be used for the base case while Lemma 
\ref{lemclutchmono} will be used for the inductive step. 

\begin{lemma}
\label{lembasecase}
Let $\ell$ be a prime distinct from $p$.  Let $U\subset \cala_g$ be an irreducible
substack which is stable under $\ell$-adic Hecke correspondences and
which is not contained in the supersingular locus.  Let $S$ be an
irreducible component of $\calm_g \cross_{\cala_g} U$.  Suppose that
$\dim(S) = \dim(U)$.  Then $\mono_\ell(S) \iso \sp_{2g}(\integ/\ell)$.
\end{lemma}

\begin{proof}
For the proof we introduce a fine moduli scheme
$\calm_{g,[N]}$, and then show that any irreducible component of
the pullback of $S$ to $\calm_{g,[N]}$ has full monodromy.

  Fix an  integer $N \ge 3$ relatively prime to $p\ell$ and an
  isomorphism $\mmu_N \iso \integ/N$, and let $\cala_{g,[N]}$ be the
  fine moduli scheme of principally polarized abelian schemes of
  relative dimension $g$ equipped with principal symplectic level-$N$
  structure \cite[p.\ 139]{mumfordgit}.  Let $\calm_{g,[N]} = \calm_g
  \cross_{\cala_g} \cala_{g,[N]}$.  It is the fine moduli scheme of
  smooth proper curves of genus $g$ with principal symplectic
  level-$N$ structure, and the induced Torelli map $\tau_{g,[N]}:
  \calm_{g,[N]} \ra \cala_{g,[N]}$ is analyzed in
  \cite{oortsteenbrink}.

  Let $U_{[N]} = U\cross_{\cala_g} \cala_{g,[N]}$.  By \cite[Prop.\
  4.4]{chailadic}, $U_{[N]}$ is irreducible and $\mono_\ell(U_{[N]})
  \iso \sp_{2g}(\integ/\ell)$.  Let $T_{[N]}$ be any irreducible
  component of $S \cross_{\cala_g} \cala_{g,[N]}$.  By \cite[Lemma
  1.11]{oortsteenbrink}, there exists an open dense subset $V_{[N]}
  \subset T_{[N]}$ and an integer $d \in \st{1,2}$ such that
 $\tau_{g,[N]}\rest{V_{[N]}}$ has degree $d$.   Moreover, if $g \le 2$
  then $d = 1$.  In summary:
\newlength{\jdahack}
\settowidth{\jdahack}{$\subseteq T_{[N]}$}
\begin{equation*}
\xymatrix{
V_{[N]}\subseteq T_{[N]} \ar[r] \ar@<-3.5ex>[d]^{\integ/d}& S \\
U_{[N]} \hspace{\jdahack}
}
\end{equation*}
Since $\dim(U) = \dim(S)$, $\tau_{g,[N]}(V_{[N]})$ is dense in
$U_{[N]}$.
By the above conditions on $d$, $\sp_{2g}(\integ/\ell)$ has no nontrivial
quotient isomorphic to $\integ/d$.
It now
follows that $\mono_\ell(V_{[N]}) \iso \sp_{2g}(\integ/\ell)$
\cite[Lemma 3.3]{achterpries07}.  
Since $\mono_\ell(V_{[N]})$ is as large as
possible, it follows that  $\mono_\ell(S) \iso
\sp_{2g}(\integ/\ell)$.
\end{proof}

\begin{lemma}
\label{lemclutchmono}
Let $g_1$, $g_2$, $r_1$ and $r_2$ be positive integers.   Let $g =
g_1+g_2$, and let $r = r_1+r_2 - 2$.
\begin{alphabetize}
\item There is a
canonical isomorphism of sheaves on $\bar \calm_{g_1;r_1} \cross \bar
\calm_{g_2;r_2}$,
\begin{equation}
\label{eqclutchsheaf}
\kappa_{g_1;r_1,g_2;r_2}^* \pic^0(\calc_{g;r})[\ell] \iso \pic^0(\calc_{g_1;r_1})[\ell]
\cross \pic^0(\calc_{g_2;r_2})[\ell].
\end{equation}

\item For $i = 1,2$, suppose $S_i \subset \bar \calm_{g_i;r_i}$
  is connected, and let $s_i \in S_i^\circ$ be a geometric point.  Let $S
  = \kappa_{g_1;r_1,g_2;r_2}(S_1 \cross S_2)$, and let $s =
  \kappa_{g_1;r_1,g_2;r_2}(s_1\cross s_2)$.  Under the identification
  \eqref{eqclutchsheaf}, 
\begin{equation*}
\mono_\ell(S_1,s_1) \cross \mono_\ell(S_2,s_2) \subseteq
\mono_\ell(S,s) \subset \aut(\pic^0(\calc_{g;r})[\ell]_s).
\end{equation*}
\end{alphabetize}
\end{lemma}

\begin{proof}
The proof of \cite[Lemma 3.1]{achterpries07} applies here.  The key
point for part (a) 
is that if $C/T$ is the union of two $T$-curves $C_1$ and $C_2$,
identified along a single section of $T$, then there is a canonical
isomorphism $\pic^0(C) \iso \pic^0(C_1) \cross \pic^0(C_2)$
\eqref{eqblr}.  The key ideas for part (b) involve the van Kampen
theorem and covariance of the fundamental group.
\end{proof}

\begin{theorem} \label{Tprankmono}
Let $\ell$ be a prime distinct from $p$ and suppose $g \ge 1$.
Suppose $0 \leq f \leq g$, and $f \not = 0$ if $g \leq 2$.
Let $S$ be an irreducible component of $\calm_g^f$, the $p$-rank $f$
stratum in $\calm_g$.
Then $\mono_\ell(S)\iso \sp_{2g}(\integ/\ell)$ and
$\mono_{\integ_\ell}(S) \iso \sp_{2g}(\integ_\ell)$.
\end{theorem}

\begin{proof}
For group-theoretic reasons (see Lemma \ref{lemintegmono}), it
suffices to show that $\mono_\ell(S) \iso \sp_{2g}(\integ/\ell)$;
for topological reasons, it suffices to show that $\mono_\ell(\bar
S)\iso \sp_{2g}(\integ/\ell)$.

The proof is by induction on $g$.  For the base cases, suppose $1 \leq
g \leq 3$ and $f\not = 0$ if $g \le 2$.  There is a Newton polygon
$\nu_g^f$ such that the locus $U^f_g$, consisting of principally
polarized abelian varieties whose Newton polygon is $\nu_g^f$, is open
and dense in $\cala_{g}^f$ \cite[Thm.\ 2.1]{oorttexelnp}.  The locus
$U^f_g$ is stable under Hecke correspondences.  Since
$f\not = 0$ if $g\le 2$, $\nu_g^f$ is not supersingular.  By
\cite[Rem.\ 4.7]{chailadic}, $U_g^f$ is irreducible. (Alternatively,
see \cite[Ex.\ 11.6]{evdg}.)
Let $S$ be any irreducible
component of $\calm_g^f$.  Then $\dim (S) = \dim (U)$ \cite[Thm.\
2.3]{FVdG:complete}; by Lemma \ref{lembasecase}, $\mono_\ell(S) \iso
\sp_{2g}(\integ/\ell)$. 

Now suppose $g \ge 4$ and $0 \le f \le g$.  As an inductive hypothesis
assume, for all pairs $(g',f')$ where $g' \ge 3$ and $0 \le f' \le g'
< g$, that $\mono_\ell(S') \iso \sp_{2g'}(\integ/\ell)$ for every
irreducible component $S'$ of $\calm_{g'}^{f'}$.

Let $S$ be an irreducible component of $\calm_g^f$.  By Corollary
\ref{cordegen}, $\overline{S}$ intersects $\Delta_{1,1}$
and there is a diagram as in \eqref{diagclutch}.  
Specifically, there 
is a partition $f = f_1+f_2+f_3$, and there are irreducible components $S_1
\subset \bar \calm^{f_1}_{1,1}$ and $S_2 \subset \bar
\calm^{f_2}_{g-2,2}$ and $S_3 \subset \bar \calm^{f_3}_{1,1}$, so that
the clutching maps in \eqref{diagclutchbox} restrict to 
\begin{equation}
\xymatrix{
S_1 \cross S_2 \cross S_3 \ar[r]  \ar[d]& 
 S_1 \cross S_R \ar[d]\\
 S_L \cross S_3 \ar[r] & \bar S \cap \Delta_{1,1}
}
\end{equation}

In particular, $S_L$ is the irreducible component of 
$\bar\calm_{g-1,1}^{f_1+f_2}$
which contains $\kappa_{1;1,g-2;1}(S_1\cross S_2)$, and $S_R$ is the
irreducible component of $\bar\calm_{g-1,1}^{f_2+f_3}$ which contains
$\kappa_{g-2;1,1;1}(S_2\cross S_3)$.  Since $g-1 \ge 3$, the inductive
hypothesis and Lemma  \ref{LLfibration} imply that $\mono_\ell(S_L) \iso \mono_\ell(S_R) \iso \sp_{2(g-1)}(\integ/\ell)$.  

Choose a base point $(s_1,s_2,s_3) \in (S_1^\circ\cross S_2^\circ \cross
S_3^\circ)(k)$, and let $s = \kappa(s_1,s_2,s_3)$.  Write $V =
\pic^0(\calc_g)[\ell]_s$ and, for $1 \leq i \leq 3$, let $V_i =
\pic^0(\calc_{g_i})[\ell]_{s_i}$.  Each of these is a
$\integ/\ell$-vector space equipped with a symplectic form.  There is
an isomorphism of symplectic $\integ/\ell$-vector spaces $V \iso V_1
\oplus V_2 \oplus V_3$ by equation \eqref{eqclutchsheaf}.  By
Lemma \ref{lemclutchmono}(b),
$\mono_\ell(S,s)$ contains $\sp(V_1\oplus V_2)\cross
\mono_\ell(S_3,s_3)$ and $\mono_\ell(S_1,s_1)\cross \sp(V_2\oplus
V_3)$.  Since $\sp(V_1\oplus V_2) \cross \sp(V_3)$ is a maximal 
subgroup of $\sp(V)$ \cite[Thm.\ 3.2]{king81}, $\mono_\ell(S,s) \iso
\sp(V) \iso \sp_{2g}(\integ/\ell)$ .
\end{proof}

\begin{remark}
The assertion of Theorem \ref{Tprankmono} is false for
$\calm^0_1$ and $\calm_2^0$ if $\ell \ge 5$.  Indeed, a
curve of genus $g\le 2$ and $p$-rank $0$ is supersingular. 
Since a
supersingular $p$-divisible group over a scheme $S$ becomes trivial
after a finite pullback $\til S \ra S$, the monodromy group
$\mono_{\integ_\ell}(\calm^0_g)$ is finite.  In view of Lemma
\ref{lemintegmono}, $\mono_\ell(\calm^0_1)$ and $\mono_\ell(\calm^0_2)$ cannot be the full symplectic group.
\end{remark}

\begin{remark}
One might wonder whether it is possible to simplify the proof of
Theorem \ref{Tprankmono}
using the intersection of $\bar S$ with both $\Delta_1$ and $\Delta_2$
rather than with $\Delta_{1,1}$.   
This does not work for the case of $\calm_4^0$.
The intersection of $\bar \calm_4^0$ with $\Delta_2$ involves curves in $\calm_2^0$, and the components of 
$\calm_2^0$ do not have full monodromy group $\sp_4(\integ/\ell)$.
\end{remark}

\subsection{The $p$-adic monodromy of the $p$-rank strata of curves}
\label{subsectionpadicmono}

In this section we calculate the $p$-adic (as opposed to the
$\ell$-adic) monodromy of components of the $p$-rank strata $\calm_g^f$.

Let $S$ be a scheme of characteristic $p$, and let $X \ra S$ be an abelian
scheme with constant $p$-rank $f$.  The group scheme
$X[p]$ admits a largest \'etale quotient $X[p]^\et$ (see, e.g, \cite[Cor.\ 1.7]{oortzink}).  
Similarly, there is a largest \'etale
quotient $X[p^\infty]^\et$ of the $p$-divisible group $X[p^\infty]$.  Let $s$ be a geometric point of $S$.
Then $X[p]^\et \ra S$ is classified by a homomorphism $\pi_1(S,s) \ra
\aut(X[p]^\et)_s \iso \gl_f(\integ/p)$, whose image is denoted
$\mono_p(X \ra S)$, the physical $p$-monodromy of $X \ra S$.
Similarly, to $X[p^\infty]^\et \ra S$ corresponds a representation
$\pi_1(S,s) \ra \aut(X[p^\infty]^\et)_s \iso \gl_f(\integ_p)$, whose
image is denoted $\mono_{\integ_p}(X \ra S)$.    

\begin{proposition}
\label{proppadic}
Suppose $g\ge 2$ and $1 \le f \le g$.  Let $S$ be an irreducible
component of $\calm_g^f$.  Then $\mono_p(S) \iso \gl_f(\integ/p)$ and
$\mono_{\integ_p}(S) \iso \gl_f(\integ_p)$.
\end{proposition}

Note that this statement is not a perfect analogue of Theorem
\ref{Tprankmono}, since it only analyzes the maximal \'etale quotient
of the $p$-divisible group of the Jacobian of the tautological curve over $S$.

\begin{proof}
If $g = f$, this is \cite[Thm.\ 2.1]{ekedahlmono}.  Otherwise, suppose
$f < g$.  By Proposition \ref{PinterDelta1}(b), there is an irreducible
component $T \subset \calm_{g-f;1}^0$ such that $\bar S$ contains the
image $Z$ of $\calm_{f;1}^f \cross T$ under
$\kappa_{f;1,g-f;1}$.  Since $\mono_p(Z)$ is isomorphic to
$\gl_f(\integ/p)$, which is as large as possible, $\mono_p(\bar S)
\iso \mono_p(S) \iso \gl_f(\integ/p)$.  Similarly,
$\mono_{\integ_p}(S) \iso \mono_{\integ_p}(Z) \iso \gl_f(\integ_p)$.
\end{proof}

\begin{remark}
In \cite{yuirredmono}, the author proves the analogue of Proposition
\ref{proppadic} for the $p$-rank strata $\cala_g^f$.  From this, he
deduces the irreducibility of a certain Igusa variety, and is able to
analyze the structure of the moduli space of abelian varieties with
specified parahoric level-$p$ structure.  The analogous statements
(e.g., \cite[Thm.\ 4.1]{yuirredmono})
hold for the moduli space of curves with parahoric level-$p$
structure, provided one replaces $\cala_g^f$ with an irreducible
component of $\calm_g^f$.
\end{remark}

\section{Arithmetic applications}
\label{secapp}

The results of the previous section about the monodromy of components of the moduli space 
of curves of genus $g$ and $p$-rank $f$ have arithmetic applications involving curves over finite fields.
Specifically, they imply
the existence of curves of a given type with trivial automorphism group
(Application \ref{apptrivaut}) or absolutely simple Jacobian
(Application \ref{appabssimp}).  Moreover, they give estimates for the
proportion of such curves with a rational point of order $\ell$ on the
Jacobian (Application  \ref{appclass}) or for which the numerator of the
zeta function has large splitting field (Application \ref{appzeta}).

For these applications, it is necessary to work over a finite
field $\ff$, as opposed to its algebraic closure.  
Throughout this section, let $\ff$ be
a finite field of characteristic $p$ and cardinality $q$, and let
$\bar\ff$ be an algebraic closure of $\ff$.
Let $\ell$ be a prime distinct from $p$.

In this section we redefine 
$\calm_g$ as the Deligne-Mumford stack of smooth projective curves of
genus $g$ fibered over the category of $\ff_p$-schemes.
 
Section \ref{subsecarithmono} contains some results on arithmetic
monodromy groups, and recalls a rigidifying structure which allows one to
pass between moduli stacks and moduli schemes.  In Sections
\ref{subsecappexist} and \ref{subsecappcount}, we apply these
techniques to deduce consequences for 
curves over finite fields.

\subsection{Arithmetic monodromy and tricanonical structures}
\label{subsecarithmono}

Let $\Lambda$ be either $\integ_\ell$, $\rat_\ell$ or $\integ/\ell^n$
for some positive integer $n$.  If $\pi: C \ra S /\ff$ is a smooth
connected proper relative curve, its geometric $\Lambda$-monodromy
group is $\mono^\geom_\Lambda(C\ra S) = \mono_\Lambda(C_{\bar\ff} \ra
S_{\bar\ff})$.  It is naturally a subgroup of the (arithmetic)
monodromy group $\mono_\Lambda(S)$.  If the fibers of $\pi$ have genus
$g$ one has $\mono_\Lambda(S) \subseteq
\gsp_{2g}(\Lambda)$, and
\begin{equation}
\label{eqmonogeom}
\mono_{\integ_\ell}(S)/\mono^\geom_{\integ_\ell}(S) \iso \gal( 
 \ff(\mmu_{\ell^\infty}(\bar \ff))/\ff).
\end{equation}

\begin{lemma}
\label{lemintegmono}
  Let $C \ra S / \ff$ be a smooth connected proper relative curve of genus $g
  \ge 2$ over a geometrically connected base.  If $\mono^\geom_\ell(S) \iso
  \sp_{2g}(\integ/\ell)$, then $\mono^\geom_{\integ_\ell}(S) \iso
  \sp_{2g}(\integ_\ell)$; $\mono_{\integ_\ell}(S)$ has finite index in
  $\gsp_{2g}(\integ_\ell)$; and $\mono^\geom_{\rat_\ell}(S) \iso
  \sp_{2g}(\rat_\ell)$.
\end{lemma}

\begin{proof}
Since $\mono^\geom_{\integ_\ell}(S)$ is a closed subgroup of
$\sp_{2g}(\integ_\ell)$, and since $\mono^\geom_{\integ_\ell}(S)
\inject \sp_{2g}(\integ_\ell) \ra\sp_{2g}(\integ/\ell)$ is surjective,
it follows that $\mono^\geom_{\integ_\ell}(S) \iso
\sp_{2g}(\integ_\ell)$ for group-theoretic reasons
\cite[Thm.\ 1.3]{vasiusurj}.  The remainder of the lemma follows from
equation \eqref{eqmonogeom} and the definition of
$\mono_{\rat_\ell}$.
\end{proof}

 In Sections \ref{subsecappexist} and \ref{subsecappcount}, we use
 Chebotarev arguments to deduce various applications about curves over
 finite fields.  At present these tools are only available for
 families of curves over schemes, as opposed to stacks.  To surmount
 this, we  consider rigidifying data whose corresponding moduli
 problems are representable by schemes.  By choosing the data
 $\calm_{g,3K}$ of a tricanonical structure which exists
 Zariski-locally on the base, as opposed to a Jacobi level structure
 which only exists \'etale-locally on the base, we can relate point
 counts on $\calm_{g,3K}(\ff)$ to those on $\calm_g(\ff)$.

Suppose $g \ge 2$.
 The canonical bundle $\Omega_{C/S}$ is ample, and $\Omega^{\tensor
   3}_{C/S}$ is very ample.  Let $N(g) = 5g-5$.  Then
 $\pi_*(\Omega_{C/S}^{\tensor 3})$ is a locally free $\calo_S$-module of
 rank $N(g)$, and sections of this bundle define a closed embedding
 $C\inject \proj^{N(g)}_S$.  A tricanonical ($3K$) structure on $C
 \ra S$ is a choice of isomorphism $\calo_S^{\oplus N(g)} \iso
 \pi_*(\Omega_{C/S}^{\tensor 3})$; let $\calm_{g,3K}$ be the moduli space
 of smooth curves of genus $g$ equipped with a $3K$-structure.  A curve
 with $3K$-structure admits no nontrivial automorphisms, and
 $\calm_{g,3K}$ is representable by a scheme \cite[10.6.5]{katzsarnak},
 \cite[Prop. 5.1]{mumfordgit}.  Moreover, $\calm_g$ may be constructed
 as the quotient of $\calm_{g,3K}$ by $\gl_{N(g)}$, so that the
 forgetful functor
 $\psi_g:\calm_{g,3K} \ra \calm_g$ is open \cite[p.\ 6]{mumfordgit} and a
 fibration \cite[Thm.\ A.12]{noohi}.   

 \begin{lemma}
\label{lem3kfibration}
 Let $S\subset \calm_g$ be a connected substack, and
 let $S_{3K} = S \cross_{\calm_g} \calm_{g,3K}$.  Then
 $\mono_{\ell}(S_{3K}) \iso \mono_{\ell}(S)$.
 \end{lemma}

 \begin{proof}
 Since $\calm_{g,3K} \ra \calm_g$ is a fibration, so is $S_{3K} \ra
 S$.  Fix a base point $s_{3K} \in S_{3K}(\bar\ff)$, and let $s =
 \psi_g(s_{3K})$.  The fiber $S_{3K,s}$ is connected; by the exact
 sequence \eqref{diagfibration}, the induced
 homomorphism $\psi_{g*}: \pi_1(S_{3K},s_{3K}) \ra \pi_1(S,s)$ is
 surjective.
 As in Lemma \ref{LLfibration}, this implies
 $\mono_\ell(S_{3K}) \iso \mono_\ell(S)$.
 \end{proof}

\subsection{Existence applications: trivial automorphism group and
simple Jacobian}
\label{subsecappexist}

In this section, we show there exist curves of genus $g$ and $p$-rank
$f$ with trivial automorphism group and absolutely simple Jacobian
using the $\rat_\ell$-monodromy of $\calm_g^f$. 

\begin{lemma}
\label{lemtrivaut}
Let $S\subset \calm_g$ be a geometrically connected substack such that
$\mono^\geom_\ell(S) \iso \sp_{2g}(\integ/\ell)$ for all $\ell$ in a
set of density one.  Then
there exists $s\in S(\bar \ff)$ such that $\aut_{\bar\ff}(\calc_{g,s})$ is
either trivial or is generated by a hyperelliptic involution.
\end{lemma}

\begin{proof}
Let $S_{3K} = S \cross_{\calm_g} \calm_{g,{3K}}$; by Lemmas
\ref{lemintegmono} and \ref{lem3kfibration}, 
$\mono^\geom_{\rat_\ell}(S_{3K}) \iso \sp_{2g}(\rat_\ell)$.  

Let $\ell$ be a prime which splits completely in all cyclotomic fields whose
  degree over $\rat$ is at most $2g$.  By the Chebotarev density theorem (see
  \cite[Cor.\ 4.3]{chailadic} for details), there exists an $s_{3K}\in
  S_{3K}(\bar\ff)$ such that $\End_{\bar\ff}(\pic^0(\calc)_{s_{3K}})\tensor\rat \iso
  L$, where $L$ is a number field of dimension $[L:\rat] = 2g$ which
  is inert at $\ell$.

  Since any automorphism of $\calc_{s_{3K},\bar\ff}$ has finite
  order, $\aut_{\bar\ff}(\calc_{s_{3K},\bar\ff})$ is contained in
  the torsion subgroup of $\calo_L\units$.  Since $L$ is linearly
  disjoint over $\rat$ from each cyclotomic field of degree at most
  $2g$, the torsion subgroup of $\calo_L\units$ is simply $\st{\pm
    1}$.  Now, $-1$ has no nontrivial fixed points on the Tate module
  $T_\ell(\pic^0(\calc_{s_{3K}}))$.  Therefore, if an automorphism
  $\iota \in \aut(\calc_{s_{3K}})$ acts as $-1$ on the Jacobian of
  $\calc_{s_{3K}}$, then the quotient of $\calc_{s_{3K}}$ by
  $\iota$ has genus zero, and $\iota$ is a hyperelliptic involution.
  For $s:=\psi_g(s_{3K}) \in S(\bar \ff)$, the group 
  $\aut_{\bar\ff}(\calc_{g,s})$ is either trivial or is generated by a
  hyperelliptic involution as well.
\end{proof}

\begin{application}
\label{apptrivaut}
  Suppose $g \ge 3$ and $0 \le f \le g$.  Then there exists
  an open dense substack $U \subset \calm^f_g$ such that for each $s \in
  U(\bar\ff)$, $\aut_{\bar\ff}(\calc_{g,s})$ is trivial.
\end{application}

\begin{proof}
  After a finite extension of the base field, one can assume that
  each irreducible component of $\calm_g^f$ is geometrically
  irreducible.  Let $S$ be one such component and recall that ${\rm dim}(S)=2g-3+f$ \cite[Thm.\ 2.3]{FVdG:complete}.
By Theorem \ref{Tprankmono}, $\mono^\geom_\ell(S) \iso \sp_{2g}(\integ/\ell)$
  for all $\ell \not = p$.  By Lemma \ref{lemtrivaut}, there is a
  nonempty (and thus open dense) substack $U_S' \subset S$ whose points
  correspond to curves of genus $g$ and $p$-rank $f$ whose
  automorphism group is either trivial or generated by a hyperelliptic
  involution.  
Every component of the $p$-rank stratum $\calh_g^f$ of the hyperelliptic locus has dimension $g-1+f$
by \cite[Thm.\ 1]{GP:05} for $p \geq 3$ and by \cite[Cor.\ 1.3]{PZ:artschprank} for $p=2$.
Thus the intersection of $U_S'$ with the hyperelliptic locus is a proper closed substack of $U'_S$. Let 
$U_S=U'_S-(U'_S \cap \calh_g^f)$, and let $U$ be the union of the stacks
$U_S$ over all irreducible components $S$; if $s\in U(\bar\ff)$, then
$\aut_{\bar\ff}(\calc_{g,s})$ is trivial.
\end{proof}

\begin{remark}
Application \ref{apptrivaut} can be proved for $\calm_g^f$ for all $g \geq 3$ and all $0 \le f \le g$ 
without monodromy techniques; see \cite[Thm.\ 1.1]{achterglasspries}. 
\end{remark}

\begin{remark}
  In special cases, there are results in the literature that have
  stronger information about the field of definition of a curve 
  with small automorphism group.  In \cite{poonennoextra1}, the author
  shows that for every $p$ and every $g \ge 3$, there is a curve of
  genus $g$ defined over $\ff_p$ with trivial automorphism group.  The
  $p$-ranks of these curves are not determined.  
\end{remark}

\begin{application}
\label{appabssimp}
Suppose $g \ge 3$ and $0 \leq f \leq g$.  Let $S$ be an irreducible component of
$\calm_g^f$.  Then there exists $s \in S(\bar\ff)$
such that the Jacobian of $\calc_s$ is absolutely simple.
\end{application}

\begin{proof} Possibly after a finite extension of $\ff$, one can assume
  that $S$ is geometrically irreducible.
By Theorem
\ref{Tprankmono}, $\mono^\geom_\ell(S) \iso \sp_{2g}(\integ/\ell)$ for all
$\ell \not = p$.  By the proof of Lemma \ref{lemtrivaut}, there exists
a point $s \in S(\bar\ff)$ such that
$\End_{\bar\ff}(\pic^0(\calc_{g,s}))\tensor\rat$ is a
field.  Then the Jacobian $\pic^0(\calc_{g,s})$ is absolutely simple.
\end{proof}

\begin{remark}
  In special cases, there are results in the literature that have
  stronger information about the field of definition of curves with
  absolutely simple Jacobians.  For example, in \cite{HZhu} the
  authors show that, for every prime $p$, if $g=2$ or $g=3$, then
  there exists a curve with genus $g$ and $p$-rank $g$ defined over
  $\ff_p$ whose Jacobian is absolutely simple.
\end{remark}

\subsection{Enumerative applications: class groups and zeta functions}

\label{subsecappcount}

This section contains enumerative results about curves of genus $g$ and $p$-rank $f$ 
that rely on $\integ/{\ell}$-monodromy groups.  
Recall that if $s \in \calm_g(\ff)$,
then $\pic^0(\calc_{g,s})(\ff)$ is isomorphic to the class group of the function field $\ff(\calc_{g,s})$.
The size of the class group is divisible by $\ell$ exactly when there is a point of order $\ell$ on the 
Jacobian.
Roughly speaking, Application \ref{appclass}
shows that among all curves over $\ff$ of specified genus and $p$-rank,
slightly more than $1/\ell$ of them have an $\ff$-rational point of
order $\ell$ on their Jacobian.

\begin{application}
\label{appclass}
Suppose $g \ge 3$, and $0 \leq f \leq g$, and $\ell$ is a prime distinct from $p$.  Let $\xi$ be the
image of $\abs \ff$ in $(\integ/\ell)\units$.  There exists a rational
function $\alpha_{g,\xi}(T)
\in \rat(T)$ such that the following holds:
there exists a constant $B= B(\calm_g^f,\ell)$ such that if
$\calm_g^f(\ff)\not =
\emptyset$, then 
\begin{equation} 
\label{eqclass}
\abs{\frac{\#\st{ s\in \calm_g^f(\ff) : \ell \text{ divides } \abs{\pic^0(\calc_{g,s})(\ff)}}}{\#
    \calm_g^f(\ff)}  - \alpha_{g,\xi}(\ell)} < \frac B{\sqrt q}.
\end{equation}
\end{application}

\begin{remark}   Suppose $\ell$ is odd.  One knows that $\alpha_{g,1}(\ell) =
\frac{\ell}{\ell^2-1} + \calo(1/\ell^3)$,
while 
$\alpha_{g,\xi}(\ell) = \oneover{\ell-1}+\calo(1/\ell^3)$ if $\xi \not
= 1$.  A formula for $\alpha_{g,1}(\ell)$ is given in
\cite{achtercl}.
\end{remark}

\begin{proof}[Proof of \ref{appclass}]
 Write $\calm$ for $\calm_g^f$.
  Let $\calm^\sm$ be the open dense locus where the reduced stack
  $\calm$ is smooth, and let $\caln$ be the union of all connected
  components $S$ of $\calm^\sm$ such that $S(\ff)\not = \emptyset$.
  Let $S$ be any such component.  
Since $S(\ff)\not = \emptyset$, $S$ is geometrically connected and smooth, and thus geometrically
  irreducible.  In particular, $S_{\bar\ff}$ is dense in an
  irreducible component of $\calm_{\bar\ff}$, and thus
  $\mono^\geom_{\ell'}(S) \iso \sp_{2g}(\integ/\ell')$ for all 
  $\ell' \not = p$.

  Let $S_{3K} = S\cross_{\calm_g} \calm_{g,{3K}}$.  Since the map
  $\psi_g: \calm_{g,{3K}} \ra \calm_g$ is a fibration with
  connected fibers, $S_{3K}$ is also connected.  Tricanonical
  structures exist Zariski-locally, so $S_{3K}(\ff)\not = \emptyset$.
  Finally, $\psi_g$ is formally smooth.  Taken together, this
  shows that $S_{3K}$ is geometrically irreducible and smooth.

  By Lemma \ref{lemtrivaut}, there is an open dense subscheme $U_{3K}$
  of $S_{3K}$ such that if $t \in U_{3K}$, then $\aut(\calc_{g,t})
  \iso \st{1}$.  The geometric monodromy group of $U_{3K}$
  is again $\sp_{2g}(\integ/\ell)$ (Lemma \ref{lem3kfibration}). An equidistribution
  theorem (\cite[9.7.13]{katzsarnak}; see also
  \cite[3.1]{achtercl}) shows that an estimate of the form \eqref{eqclass}
  holds (with error term of order $\calo(1/\sqrt q)$), where $\calm$ is replaced by $U_{3K}$.  Let $U =
  \psi_g(U_{3K})$; since $\psi_g$ is
  universally open, $U$ is open, 
  too.  The fiber over each $s\in U(\ff)$ consists of exactly
  $\abs{\gl_{N(g)}(\ff)}/\abs{\aut(\calc_{g,s}))} = \abs{\gl_{N(g)}(\ff)}$ points 
  \cite[10.6.8]{katzsarnak}, and if $\psi_g(t) = s$ then $\calc_{g,t}
  \iso \calc_{g,s}$.  Therefore, the proportion of elements $s \in
  U(\ff)$ for which $\ell$ divides $\abs{\pic^0(\calc_{g,s})(\ff)}$ is exactly
  the same as the analogous proportion of elements $t \in
  U_{3K}(\ff)$.

  Thus \eqref{eqclass} holds when $\calm$ is replaced by $U$.  For
  dimension reasons, there exists a constant $D$ such that $\#(S-
  U)(\ff)/\#S(\ff) < D/q$; therefore, \eqref{eqclass} holds for $S$.
  By invoking this argument for each of the finitely many irreducible
  components of $\caln$, and remembering that by construction
  $\caln(\ff) = \calm^\sm(\ff)$, one obtains \eqref{eqclass} for
  $\calm^\sm(\ff)$.  Finally, since there exists a constant $D'$ such
  that $\#(\calm - \calm^\sm)(\ff)/\#\calm(\ff) < D'/q$, this yields
  \eqref{eqclass}.
\end{proof}

If $C/\ff$ is a smooth projective curve of genus $g$, its zeta function has the form
$P_{C/\ff}(T)/(1-T)(1-qT)$, where $P_{C/\ff}(T) \in \integ[T]$ is a polynomial
of degree $2g$.  The principal polarization on the Jacobian of $C$
forces a symmetry among the roots of $P_{C/\ff}(T)$; the largest
possible Galois group for the splitting field over $\rat$ of
$P_{C/\ff}(T)$ is the Weyl group of $\sp_{2g}$ which is a group of
size $2^gg!$.

\begin{application}
\label{appzeta}
Suppose $g \ge 3$, and $0 \leq f \leq g$, and $p>2g+1$.  There exists a constant $\gamma=
\gamma(g)>0$ so that the 
following holds. 
There exists a constant $E = E(\calm_g^f)$ so that if $\calm_g^f(\ff)\not = \emptyset$,
then 
\begin{equation}
\label{eqzeta}
\frac{\# \st{ s \in \calm_g^f(\ff) : P_{\calc_{g,s}/\ff}(T)\text{ is reducible, or has
      splitting field with degree }< 2^gg!}}{\#\calm_g^f(\ff)} < E q^{-\gamma}.
\end{equation}
\end{application}

\begin{proof}
  The proof is similar to that of Application \ref{appclass}.  Again, write $\calm$ for $\calm_g^f$.
  let $S$ be an irreducible component of $\calm^\sm$ with $S(\ff)\not
  = \emptyset$, and let $U_{3K}$ be the open dense subscheme of
  $S_{3K}$ whose points correspond to curves of the specified type with
  tricanonical structure with trivial automorphism group.  By Lemma \ref{lem3kfibration}, and Theorem
  \ref{Tprankmono}, $\mono^\geom_\ell(U_{3K}) \iso \sp_{2g}(\integ/\ell)$
  if $\ell \not =p$.  By \cite[Thm. 6.1
  and Remark 3.2.(4)]{kowalskisieve}, there is a constant
  $E(S)$ so that \eqref{eqzeta} is valid for $U_{3K}$.  The argument
  used in Application \ref{appclass} shows the same result for each
  $S$, and thus for $\calm$.
\end{proof}

\thanks{The second author was partially supported by NSF grant DMS-07-01303.}

\bibliographystyle{abbrv} 
\bibliography{prm}

\end{document}